\documentclass{amsart}

	\usepackage{graphicx}
	\usepackage{amssymb}
	\usepackage[all]{xy} 

	\theoremstyle{plain}
	\newtheorem{theorem}{Theorem}[section]
	
	\newtheorem{lemma}[theorem]{Lemma}
	\newtheorem{corollary}[theorem]{Corollary}

	\theoremstyle{definition}
	\newtheorem{definition}[theorem]{Definition}

	\theoremstyle{remark}
	\newtheorem{remark}[theorem]{Remark}

	\theoremstyle{example}
	\newtheorem{example}[theorem]{Example}

	\theoremstyle{conjecture}

	\DeclareMathOperator{\tr}{tr}
    \newcommand{\0}{\mathbf{0}}

	\graphicspath{{./graphics/}}

\begin{document}

\title[The Square of Adjacency Matrices]
{The Square of Adjacency Matrices}
\author{Daniel Joseph Kranda}
\date{\today}
\maketitle

\begin{abstract}

It can be shown that any symmetric $(0,1)$-matrix $A$ with $\tr A = 0$ 
can be interpreted as the adjacency matrix of
a simple, finite graph. The square of an adjacency matrix $A^2=(s_{ij})$ 
has the property that $s_{ij}$ represents the number
of walks of length two from vertex $i$ to vertex $j$.
With this information, the motivating question behind this paper
was to determine what conditions on a matrix $S$ are needed to have
$S=A(G)^2$ for some graph $G$. Structural results imposed by the
matrix $S$ include detecting bipartiteness or connectedness and 
counting four cycles. Row and column sums are examined as well as
the problem of multiple nonisomorphic graphs with the same adjacency
matrix squared.

\end{abstract}

\section{Introduction and Background}

This paper aims to determine properties of graphs found by examining the square of the adjacency matrix and vice versa.
As a thorough study of the square of the adjacency matrix has not been adressed
previously in the literature, we survey several results covering different aspects of the given problem.

Throughout this paper, we will consider only simple, undirected graphs.

\begin{theorem}
Let $A=(a_{ij})=A(G)$ for some simple undirected graph $G$ and define $S=(s_{ij})=A^2$. 
Then for every $i$ and $j$, $s_{ij}$ represents the number of
two-walks (walks with two edges) from vertex $v_i$ to $v_j$ in $G$.
\end{theorem}

\begin{proof}
Consider the entry $s_{ij}$ in $S$. By definition, $s_{ij}=\sum_{k=1}^{n}{a_{ik}a_{kj}}$ and so one is contributed to the sum only when
$a_{ik}$ and $a_{kj}$ are $1$. That is, when the edges $v_iv_k$ and $v_kv_j$ are in $G$, 
which corresponds to the two-walk from $v_i$ to $v_j$ through $v_k$.
\end{proof}

\begin{definition}
A matrix $S$ is \emph{square graphic} if there is a simple, undirected graph $G$ such that $S=A(G)^2$.
\end{definition}

\begin{definition}\label{square similar def}
We will say $S_1$ and $S_2$ are \emph{similar} if there is a permutation matrix $P$ 
such that $S_2=P^{-1}S_1P$. In this case, we write $S_1\sim S_2$.
\end{definition}

\begin{theorem}\label{square similar thm}
If $S_1$ is square graphic, then so is $S_2=P^{-1}S_1P$ for any permutation matrix $P$.
\end{theorem}

\begin{remark} If $S_1\sim S_2$ then $S_1$ is square graphic if and only if $S_2$ is square graphic. \end{remark}

The following results relate the structure of the square of the adjacency matrix of a graph with the structure of that graph.
We begin by determining if a graph is bipartite or disconnected by examining the adjacency matrix squared.

\begin{theorem}\label{bip or disc}
Suppose $S$ is an $n\times n$ matrix such that $S=A(G)^2$. Then $G$ is bipartite or disconnected if and only if
$S\sim\left(\begin{array}{cc} B_1 & \0 \\ \0 & B_2 \end{array}\right)$ where $B_1$ is a $k\times k$ matrix with $0<k<n$ and
$\0$ represents a matrix of all zeros of the appropriate size.
\end{theorem}

\begin{proof}
First suppose that $G$ is disconnected with a connected component $H$ on $k$ vertices with $0<k<n$. Then there are no two-walks from
any vertex of $H$ to any vertex of $G\setminus H$. Therefore, renumbering the vertices of $G$ if necessary, we have
\[A(G)^2\sim\left(\begin{array}{cc} A(H)^2 & \0 \\ \0 & A(G\setminus H)^2 \end{array}\right).\]
It is clear $A(G)^2$ has the desired form.

Now, suppose $G$ is bipartite with partite sets $X$ and $Y$ and that $A(G)^2=(s_{ij})$.
Without loss of generality, we have $X=\{v_1,v_2,\ldots,v_k\}$ (otherwise, relabel the graph accordingly).
Since $G$ is bipartite, every two-walk must begin and end in the same partite set.
Hence, $s_{ij}=s_{ji}=0$ for all $i=1,2,\ldots,k$ and $j=k+1,\ldots, n$. Therefore,
\[A(G)^2=\left(\begin{array}{cc} B_1 & \0 \\ \0 & B_2 \end{array}\right)\]
where $B_1$ is $k\times k$ with $0<k<n$.

To prove the sufficiency of the statement, assume by the contrapositive that $G$ is connected and nonbipartite. 
By contradiction, assume
\[A(G)^2=(s_{ij})=\left(\begin{array}{cc} B_1 & \0 \\ \0 & B_2 \end{array}\right)\]
where $B_1$ is $k\times k$ with $0<k<n$.
Let $V_1=\{v_1,v_2,\ldots, v_k\}$ and $V_2=\{v_{k+1},\ldots, v_n\}$. Under these assumptions,
we must have $s_{ij}=s_{ji}=0$ for all $i=1,\ldots, k$ and $j=k+1,\ldots, n$.

Since $G$ is nonbipartite, there is an odd cycle $C$ in $G$ of length $t$. Without loss of generality, there is a $v\in V(C)\cap V_1$.
We now claim that $V(C)\subseteq V_1$.

Write $C= v_{c_0}v_{c_1}\ldots v_{c_{t-1}} v_{c_0}$ where the indices are $c_i$ with $i \mod t$. Then without loss of generality,
$v=v_{c_0}\in V_1$. Notice that  we must have $v_{c_{i+2}}\in V_1$ whenever $v_{c_i}\in V_1$. Otherwise, if $v_{c_i}\in V_1$ and
$v_{c_{i+2}}\in V_2$ then $s_{c_ic_{i+2}}\neq 0$ which is a contradiction, 
because ${c_i}\in\{1,\ldots, k\}$ and $c_{i+2}\in\{k+1,\ldots, n\}$.
Therefore, $v_{c_{2p}}\in V_1$ for $p=0,1,2,\ldots, t-1$. 
But since $C$ is of odd length, this forces $V(C)\subseteq V_1$, proving the claim.

Now, if there is a vertex $u\in V(G\setminus C)$ then since $G$ is assumed to be connected, 
there exists a path $P$ from $u$ to a vertex $v$ on $C$
so that $P\cap C = \{v\}$.

If $P$ is of even length, then since every second vertex from $v$ on $P$ must also be in $V_1$, 
we have that $u\in V_1$.

If $P$ is of odd length, consider a neighbor $w$ of $v$, such that $w\in V(C)\subseteq V_1$. 
Then $wvPu$ is a path of even length, and the argument from above forces $u\in V_1$.

Therefore, every vertex of $G$ must be in $V_1$, making $|V_1|=n$ and $|V_2|=0$ which is a contradiction. 
Therefore, we have proven the claim by contrapositive. That is,
if $G$ is nonbipartite and connected, then $A(G)^2$ is not similar
to a block diagonal matrix.
\end{proof}

We now give a result on counting four cycles in a graph.

\begin{theorem}\label{c4s}
If $S=(s_{ij})=A(G)^2$ for some graph $G$, then
\[\frac{1}{4}\sum_{i\neq j}{s_{ij} \choose 2}\]
is the number of distinct cycles of length four in $G$.
\end{theorem}

\begin{proof}
First, we claim that, for $i\neq j$, ${s_{ij} \choose 2}$ counts the number of distinct cycles of 
length four on which vertices $v_i$ and $v_j$ sit opposite.
To prove the claim, let $v_i,v_j\in V(G)$ and notice every two-walk from $v_i$ to $v_j$ 
corresponds to a shared neighbor of the two. Now, a cycle of length
four on which $v_i$ and $v_j$ sit opposite occurs when there is a 
two-walk from $v_i$ to $v_j$ and  a different two-walk from $v_j$ to $v_i$. In other words,
$v_i$ and $v_j$ sit opposite on a cycle of length four when we can choose 
two distinct vertices $u$ and $v$ that are neighbors of both $v_i$ and $v_j$. Since the number
of shared neighbors of $v_i$ and $v_j$ is exactly $s_{ij}$, the number of cycles of 
length four on which $v_i$ and $v_j$ sit opposite is ${s_{ij} \choose 2}$.

Consider a cycle of length four in $G$: $uvwxu$. In the sum $\sum_{i\neq j}{s_{ij} \choose 2}$, 
this cycle is counted once by each of ${s_{uw} \choose 2},{s_{wu} \choose 2},
{s_{vx} \choose 2}$, and ${s_{xv} \choose 2}$. Thus, to count each four cycle in $G$ exactly once, 
we divide this sum by four.
\end{proof}

\begin{remark}
A necessary condition for a matrix $S$ to be square graphic that can be taken from 
Theorem \ref{c4s} is that the number $\sum_{i\neq j}{s_{ij} \choose 2}$ must be divisible by four.
\end{remark}

\section{Row and Column Sums}

We now prove some results relating the row and column sums of the square of the adjacency matrix to
properties of the underlying graph.

\begin{definition}
The \emph{neighborhood of $v$}, $\Gamma(v)$, is the set of all vertices in $G$ adjacent to $v$.
\end{definition}

\begin{theorem}\label{rowsum}
If $S=(s_{ij})=A(G)^2$ for some graph $G$ then
\[\sum_{j=1}^{n} s_{ij} = \sum_{j=1}^{n} s_{ji} = \sum_{v\in\Gamma(v_i)} \deg(v)\]
and thus, if $s_{ii}\neq 0$
\[\frac{1}{s_{ii}}\sum_{j=1}^{n} s_{ij}\]
gives the average degrees of the neighbors of $v_i$.
\end{theorem}

\begin{proof}
Consider $v_i\in G$ and some $v\in\Gamma(v_i)$. Then there are exactly $\deg v$ two-walks of the form $v_i v u$. 
Since every two-walk starting at $v_i$ must go through some neighbor of $v_i$, by taking the sum of the degrees of the
neighbors of $v_i$, we will have counted all possible two-walks from $v_i$. On the other hand, 
$\sum_j s_{ij}$ gives the total number of two-walks starting at $v_i$. Thus,
\[\sum_{j=1}^{n} s_{ij} = \sum_{v\in\Gamma(v_i)} \deg(v).\]

The number of summands on the right hand side is exactly $|\Gamma(v_i)|=\deg(v_i)=s_{ii}$. Dividing across gives the desired result.
\end{proof}

\begin{corollary}\label{cor1 rowsum}
If $S=(s_{ij})=A(G)^2$ for some graph $G$ then for each $i$ there is $E_i\subseteq\{s_{11},s_{22},\ldots,s_{nn}\}\setminus\{s_{ii}\}$ 
(viewed as a multiset if necessary) such that $|E_i|=s_{ii}$ and
\[\sum_{s\in E_i}{s}=\sum_{j=1}^{n}{s_{ij}}.\]
\end{corollary}

\begin{proof}
We have
\[\sum_{j=1}^{n} s_{ij} = \sum_{v\in\Gamma(v_i)} \deg(v)\]
and since $|\Gamma(v_i)|=\deg(v_i)=s_{ii}$, the number of summands on the right hand side of this equation is $s_{ii}$. 
For each $v_j\in\Gamma(v_i)$, we have $\deg(v_j)=s_{jj}$.
Taking $E_i=\{s_{jj} $ such that $v_j\in\Gamma(v_i)\}$ gives the desired result.
\end{proof}

\begin{corollary}\label{cor2 rowsum}
If $S=(s_{ij})=A(G)^2$ where $G$ is a $k$-regular graph, then
\[\sum_{j=1}^{n} s_{ij}=\sum_{j=1}^{n} s_{ji} = k^2.\]
\end{corollary}

\begin{proof}
We have
\[\sum_{j=1}^{n} s_{ij} = \sum_{v\in\Gamma(v_i)} \deg(v) = \sum_{v\in\Gamma(v_i)} k = k^2\]
since $|\Gamma(v_i)|=\deg(v_i)=k$ for all $i$.
\end{proof}

It should be noted that,
the previous results can be used to determine if a matrix $S$ is square graphic as shown 
in the next example.

\begin{example}
Consider the matrix
\[S=\left(\begin{array}{cccc}
2 & 1 & 1 & 0 \\
1 & 2 & 1 & 1 \\
1 & 1 & 1 & 0 \\
0 & 1 & 0 & 1 
\end{array}\right).\]
Then by the previous results, if $S$ were square graphic, then the average degree of the neighbors
of $v_2$ would be
\[\frac{1}{s_{22}}\sum_{i=1}^{4}s_{2i}=\bigg(\frac{1}{2}\bigg)(5)=\frac{5}{2}.\]
This implies that $v_2$ must have a neighbor of degree at least 3, 
which is impossible given the diagonal of $S$.
Therefore, $S$ cannot be square graphic.
\end{example}

\section{Duplication}

Determining when a given matrix was square graphic lead to the interesting problem 
of determining when a matrix represented the square of
the adjacency matrix of several non-isomorphic graphs.
For example, the graphs in Figures \ref{2C3s} and \ref{C6} are non isomorphic
with the same adjacency matrix squared.

\begin{figure}[h!]
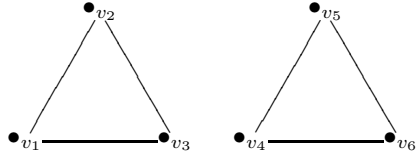

\[\xygraph{
!{<0cm,0cm>;<1cm,0cm>:<0cm,1cm>::}
!{(0,0) }*{\bullet_{v_1}}="v_1"
!{(1,1.73) }*{\bullet_{v_2}}="v_2"
!{(2,0) }*{\bullet_{v_3}}="v_3"
!{(3,0)}*{\bullet_{v_4}}="v_4"
!{(4,1.73)}*{\bullet_{v_5}}="v_5"
!{(5,0)}*{\bullet_{v_6}}="v_6"
"v_1"-"v_2" "v_2"-"v_3" "v_3"-"v_1"
"v_4"-"v_5" "v_5"-"v_6" "v_6"-"v_4"
} \]
\caption[55mm]{\label{2C3s}$C_3\cup C_3$}
\end{figure}

\begin{figure}[h!]
\[\xygraph{
!{<0cm,0cm>;<1cm,0cm>:<0cm,1cm>::}
!{(0,0) }*{\bullet_{v_1}}="v_1"
!{(1,1.73) }*{\bullet_{v_5}}="v_5"
!{(3,1.73) }*{\bullet_{v_3}}="v_3"
!{(4,0)}*{\bullet_{v_4}}="v_4"
!{(3,-1.73)}*{\bullet_{v_2}}="v_2"
!{(1,-1.73)}*{\bullet_{v_6}}="v_6"
"v_1"-"v_5" "v_5"-"v_3" "v_3"-"v_4"
"v_4"-"v_2" "v_2"-"v_6" "v_6"-"v_1"
} \]
\caption[45mm]{\label{C6}$C_6$}
\end{figure}

The following theorem serves as a starting point for the construction of 
squares of adjacency matrices corresponding to several non-isomorphic graphs.

\begin{theorem}\label{bip copy}
We have
\[\left(\begin{array}{cc} 
A(G) & \0 \\ 
\0 & A(G) \end{array}\right)
\sim\left(\begin{array}{cc} 
\0 & A(G) \\ 
A(G) & \0 \end{array}\right)\]
if and only if
$G$ is bipartite.
\end{theorem}

Before the proof of this theorem is given, we recall the following fact from graph theory.

\begin{lemma}\label{bip adj}
We have $G$ is bipartite if and only if 
$A(G)\sim\left(\begin{array}{cc}
\0 & B^T \\
B & \0\end{array}\right)$.
\end{lemma}

\begin{proof}[Proof of Theorem \ref{bip copy}.]
Suppose $G$ is bipartite. Then 
$A(G)=\left(\begin{array}{cc} 
\0 & B^T \\ 
B & \0 \end{array}\right)$ 
for some $m\times n$ matrix $B$.
Consider the permutation matrix
\[P=\left(\begin{array}{cccc} 
\0 & \0 & I_n & \0 \\ 
\0 & I_m & \0 & \0 \\ 
I_n & \0 & \0 & \0 \\ 
\0 & \0 & \0 & I_m \end{array}\right).\]
Then we have
\[P^{-1}\left(\begin{array}{cccc} 
\0 & \0 & \0 & B^T \\ 
\0 & \0 & B & \0 \\ 
\0 & B^T & \0 & \0 \\ 
B & \0 & \0 & \0 \end{array}\right)P
=
\left(\begin{array}{cccc} 
\0 & B^T & \0 & \0 \\ 
B & \0 & \0 & \0 \\ 
\0 & \0 & \0 & B^T \\ 
\0 & \0 & B & \0 \end{array}\right)\]
and so
\[\left(\begin{array}{cc} 
A(G) & \0 \\ 
\0 & A(G) \end{array}\right)
\sim\left(\begin{array}{cc} 
\0 & A(G) \\ 
A(G) & \0 \end{array}\right).\]

On the other hand, suppose $G$ is nonbipartite. Then
\[\left(\begin{array}{cc} A(G) & \0 \\ \0 & A(G) \end{array}\right)=A(G\cup G)\]
and thus, is the adjacency matrix of a disconnected, nonbipartite graph $H_1$. On the other hand, if $B=A(G)$ then
\[\left(\begin{array}{cc} \0 & A(G) \\ A(G) & \0 \end{array}\right)=\left(\begin{array}{cc} \0 & B^T \\ B & \0 \end{array}\right)\]
and thus, is the adjacency matrix of a bipartite graph $H_2$ by Lemma \ref{bip adj}. 
Therefore, $H_1\not\cong H_2$ and hence, $A(H_1)\not\sim A(H_2)$ by Theorem \ref{similar adj}.
\end{proof}

\begin{remark}\label{bip copy remark}
By the previous theorem, given any nonbipartite graph $G$, the graphs whose adjacency matrices are
\[A(H_1)=\left(\begin{array}{cc} A(G) & \0 \\ \0 & A(G) \end{array}\right) 
\text{ and } 
A(H_2)=\left(\begin{array}{cc} \0 & A(G) \\ A(G) & \0 \end{array}\right)\]
are non-isomorphic graphs with $A(H_1)^2=A(H_2)^2$.

It should be noted that $H_1\cong G\cup G$ and $H_2$ is known as the bipartite double cover graph of $G$ 
or the Kronecker cover of $G$.
\end{remark}

This result can be used to build matrices $S$ with arbitrarily many non-isomorphic graphs whose adjacency matrix squared is $S$.
This process is described in the following theorem.

\begin{theorem}
For every positive integer $k$ and integer $n\ge 3$, 
there exists a matrix $S$ of size $(2kn)\times(2kn)$ such that $A(G_i)^2=S$ for $k+1$ non-isomorphic graphs $G_1,G_2,\ldots, G_{k+1}$.
\end{theorem}

\begin{proof}
Let $G$ be a nonbipartite graph on $n$ vertices. Note, $n\ge 3$ since we must have an odd cycle 
in $G$ the smallest of which is length 3.
Let $A$ be the block diagonal matrix with $2k$ copies of $A(G)$ on the main block-diagonal.
That is,
\[A=\left(\begin{array}{cccc} 
A(G) & \0 & \cdots & \0 \\
\0 & A(G) & &\vdots \\
\vdots & & \ddots & \0 \\
\0 & \cdots & \0 & A(G)
\end{array}\right).\]
If we define $S=A^2$, then $S$ is square graphic since $S=A(\bigcup_{i=1}^{2k}{G})^2$.

Let $H$ be the bipartite double cover graph of $G$; that is, the graph $H$ such that 
\[A(H)=\left(\begin{array}{cc} \0 & A(G) \\ A(G) & \0 \end{array}\right)\]
and define the permutation $\pi_{2t}=(12)(34)\cdots(2(t-1) \ 2t)$  for each $t=1,2,\ldots,k$.
For each permutation, let $P_{\pi_{2t}}$ be the block permutation matrix of size $(2kn)\times (2kn)$ 
swapping $n$ rows of $I_{2kn}$ at a time according to the permutation $\pi_{2t}$.

For example,
\[P_{\pi_2}=
\left(\begin{array}{ccccc}
\0 & I_n & \0 & \cdots & \0 \\
I_n & \0 & \0 & & \vdots \\
\0 & \0 & I_n & & \vdots\\
\vdots & & & \ddots & \0 \\
\0 & \cdots & \cdots & \0 & I_n
\end{array}\right).\]

Then, for every $t=1,2,\ldots,k$ we have
\[P_{\pi_{2t}}A=
\left(\begin{array}{ccccccc} 
A(H)   & \0   &\cdots& \0   &  \0  &\cdots& \0   \\
\0     & A(H) &\ddots&\vdots&\vdots&      &\vdots\\
\vdots &\ddots&\ddots& \0   &\vdots&      &\vdots\\
\0     &\cdots& \0   & A(H) & \0   &\cdots& \0   \\
\0     &\cdots&\cdots& \0   & A(G) &\ddots&\vdots\\
\vdots &      &      &\vdots&\ddots&\ddots& \0   \\
\0     &\cdots&\cdots& \0   &\cdots& \0   & A(G)
\end{array}\right)\]
where there are $t$ copies of $A(H)$ and $2(k-t)$ copies of $A(G)$ on the main block diagonal.
Next, define the graphs $G_t$ by
\[A(G_t)=P_{\pi_{2t}}A=A\big((\bigcup_{i=1}^{t}H)\cup(\bigcup_{j=1}^{2(k-t)}G)\big).\]

Since $G$ is nonbipartite, we have $G_i\not\cong G_j$ for $i\neq j$; 
however, $A(G_t)^2=S$ for all $t=1,2,\ldots, k$ by Theorem \ref{bip copy} and Remark \ref{bip copy remark}.

Therefore, $S$ is $(2kn)\times(2kn)$ and the square of the adjacency matrix for the $k+1$ non-isomorphic graphs:
$G_1, \ldots, G_{k-1},G_k$ and $\bigcup_{i=1}^{2k}{G}$.
\end{proof}

\begin{remark}
There are nonbipartite, connected, non-isomorphic graphs whose adjacency matrices squared are similar.
\end{remark}

\begin{example}
The graphs $G$ and $H$ from Figures \ref{G} and \ref{H}, respectively, are nonbipartite, connected, non-isomorphic graphs
whose adjacency matrices squared are similar. Note that in each graph, the vertices labeled $v_1$ are identified; and so, $G$
and $H$ are both $4$-regular. These graphs were found in \cite{cvetkovic:spectra}.

\begin{figure}[h!]
\[\xygraph{
!{<0cm,0cm>;<1cm,0cm>:<0cm,1cm>::}
!{(0,0) }*{\bullet_{v_1}}="v1"
!{(.707,.707) }*{\bullet_{v_2}}="v2"
!{(.707,-.707) }*{\bullet_{v_3}}="v3"
!{(1.414,0)}*{\bullet_{v_4}}="v4"
!{(2.121,.707)}*{\bullet_{v_5}}="v5"
!{(2.121,-.707)}*{\bullet_{v_6}}="v6"
!{(2.828,0)}*{\bullet_{v_7}}="v7"
!{(3.536,.707)}*{\bullet_{v_8}}="v8"
!{(3.536,-.707)}*{\bullet_{v_9}}="v9"
!{(4.243,0)}*{\bullet_{v_{10}}}="v10"
!{(4.950,.707)}*{\bullet_{v_{11}}}="v11"
!{(4.950,-.707)}*{\bullet_{v_{12}}}="v12"
!{(5.657,0) }*{\bullet_{v_1}}="v1a"
"v1"-"v2" "v1"-"v3" "v2"-"v4" "v3"-"v4" "v4"-"v5" "v4"-"v6" "v2"-"v5" "v3"-"v6"
"v5"-"v7" "v5"-"v8" "v6"-"v7" "v7"-"v8" "v7"-"v9" "v6"-"v9" "v8"-"v10" "v10"-"v11" "v10"-"v12"
"v8"-"v11" "v9"-"v10" "v9"-"v12" "v11"-"v1a" "v12"-"v1a" "v2"-@/^.5cm/"v11" "v3"-@/_.5cm/"v12"
} \]
\caption[60mm]{\label{G} Graph $G$}
\end{figure}

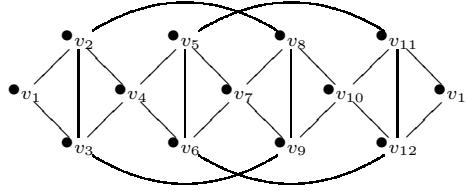
\begin{figure}[h!]
\[\xygraph{
!{<0cm,0cm>;<1cm,0cm>:<0cm,1cm>::}
!{(0,0) }*{\bullet_{v_1}}="v1"
!{(.707,.707) }*{\bullet_{v_2}}="v2"
!{(.707,-.707) }*{\bullet_{v_3}}="v3"
!{(1.414,0)}*{\bullet_{v_4}}="v4"
!{(2.121,.707)}*{\bullet_{v_5}}="v5"
!{(2.121,-.707)}*{\bullet_{v_6}}="v6"
!{(2.828,0)}*{\bullet_{v_7}}="v7"
!{(3.536,.707)}*{\bullet_{v_8}}="v8"
!{(3.536,-.707)}*{\bullet_{v_9}}="v9"
!{(4.243,0)}*{\bullet_{v_{10}}}="v10"
!{(4.950,.707)}*{\bullet_{v_{11}}}="v11"
!{(4.950,-.707)}*{\bullet_{v_{12}}}="v12"
!{(5.657,0) }*{\bullet_{v_1}}="v1a"
"v1"-"v2" "v1"-"v3" "v2"-"v4" "v3"-"v4" "v4"-"v5" "v4"-"v6"
"v5"-"v7" "v6"-"v7" "v7"-"v8" "v7"-"v9" "v8"-"v10" "v10"-"v11" "v10"-"v12"
"v9"-"v10" "v11"-"v1a" "v12"-"v1a" "v2"-@/^.5cm/"v8" "v3"-@/_.5cm/"v9"
"v5"-@/^.5cm/"v11" "v6"-@/_.5cm/"v12"
"v2"-"v3" "v5"-"v6" "v8"-"v9" "v11"-"v12"
} \]
\caption[60mm]{\label{H} Graph $H$}
\end{figure}
\end{example}

With this example, it appears to the author that the problem of duplication is more complicated than
initially suspected and will require further study.

\bibliographystyle{amsplain}
\bibliography{pubbib}

\end{document}